\documentclass{book}

\usepackage[contrib,lang=english]{ems-book2}
\usepackage{newtxtext,newtxmath}
\usepackage{mathrsfs}
\usepackage{setspace}
\newcommand{\R}{{\mathbb R}}
\newcommand{\N}{{\mathbb N}}

\newcommand{\be}[1]{\begin{equation}\label{#1}}
\newcommand{\ee}{\end{equation}}
\renewcommand{\(}{\left(}
\renewcommand{\)}{\right)}
\newcommand{\irn}[1]{\int_{\R^n}{#1}\,dx}

\newcommand{\nrm}[2]{\|{#1}\|_{L^{#2}(\R^n)}}
\newcommand{\nrmsd}[2]{\|{#1}\|_{L^{#2}(\mathbb S^n)}}

\DeclareFontEncoding{FMS}{}{}
\DeclareFontSubstitution{FMS}{futm}{m}{n}
\DeclareFontEncoding{FMX}{}{}
\DeclareFontSubstitution{FMX}{futm}{m}{n}
\DeclareSymbolFont{fouriersymbols}{FMS}{futm}{m}{n}
\DeclareSymbolFont{fourierlargesymbols}{FMX}{futm}{m}{n}
\DeclareMathDelimiter{\VERT}{\mathord}{fouriersymbols}{152}{fourierlargesymbols}{147}
\newcommand{\nnrm}[1]{\VERT{#1}\VERT}

\newtheorem{thm}{Theorem}[section]

\newtheorem{prop}[thm]{Proposition}

\theoremstyle{definition}

\numberwithin{equation}{section}

\makeatletter
\def\@makechapterhead#1{\vspace*{50\p@}{\parindent \z@ \raggedright \normalfont\interlinepenalty\@M\Large \bfseries #1\par\nobreak\vskip 40\p@}}
\def\@makeschapterhead#1{\vspace*{50\p@}{\parindent \z@ \raggedright\normalfont\interlinepenalty\@M\Large \bfseries #1\par\nobreak\vskip 40\p@}}
\renewcommand{\thesection}{\@arabic\c@section}
\renewcommand{\thesubsection}{\@arabic\c@section.\@arabic\c@subsection}
\makeatother

\usepackage{etoolbox}
\newcounter{taggedeq}
\pretocmd{\equation}{\stepcounter{taggedeq}}{}{}

\begin{document}
\mainmatter
\title{Hardy-Littlewood-Sobolev and related inequalities: stability}
\titlemark{Hardy-Littlewood-Sobolev inequalities: stability}

\emsauthor{1}{J.~Dolbeault, M.~J.~Esteban}{J.~Dolbeault, M.~J.~Esteban}
\emsaffil{1}{CEREMADE (CNRS UMR n$^\circ$ 7534), PSL university, Universit\'e Paris-Dauphine, Place de Lattre de Tassigny, 75775 Paris 16, France \email{dolbeaul@ceremade.dauphine.fr, esteban@ceremade.dauphine.fr}}

\classification[26D10; 49K30; 46E35; 35J20; 49K20]{49J40}

\begin{abstract}
The purpose of this chapter is twofold. We present a review of the existing stability results for Sobolev, Hardy-Littlewood-Sobolev (HLS) and related inequalities. We also contribute to the topic with some observations on constructive stability estimates for~(HLS).
\end{abstract}

\makecontribtitle

%%%%%%%%%%%%%%%%%%%%%%%%%%%%%%%%%%%%%%%%%%%%%%%%%%%%%%%%%%%%%%%%%%%%%%
%%%%%%%%%%%%%%%%%%%%%%%%%%%%%%%%%%%%%%%%%%%%%%%%%%%%%%%%%%%%%%%%%%%%%%
\emph{It is with great pleasure that we dedicate this paper to Elliott Lieb on the occasion of his 90th birthday.}

\section{A short review of some functional inequalities}\label{Sec:Introduction-old}

Functional inequalities play a very important role in various fields of mathematics, ranging from geometry, analysis, and probability theory to mathematical physics. For many problems, the precise value of the best constants matters and was actively studied, often in relation with the explicit knowledge of the optimizers. A standard scheme goes as follows: by rearrangement and symmetrisation, optimality is reduced to a smaller class of functions, for instance, to radial functions. After proving that the equality case is achieved, the Euler-Lagrange equations are solved by ODE techniques, which allows to classify the optimal functions and compute the best constants. This is the strategy of E.H.~Lieb in~\cite{Lieb-1983} for the \emph{Hardy-Littlewood-Sobolev inequality} which can be written as
\be{HLSgen}
\iint_{\R^n\times\R^n}f(x)\,|x-y|^{-\lambda} g(y)\,dx\,dy\le C_{p,\lambda,n}\,\|f\|_{L^p(\R^n)}\,\|g\|_{L^t(\R^n)}\,,\tag{$\mathrm{HLS}_\lambda$}
\ee
for all $f\in L^p(\R^n)$, $g\in L^t(\R^n)$, with $1<p$, $t<+\infty$ such that $1/p+1/t+\lambda/n=2$ and $0<\lambda<n$. By duality he also obtained a new and simple proof of the \emph{fractional Sobolev inequality}, which goes as follows. If $\alpha\in(0,n/2)$ and $q=2\,n/(n-2\alpha)$, then
\be{Sobfrac}
\nrm{(-\Delta)^{\alpha/2}f}2^2\ge S_{n,\alpha}\,\nrm fq^2\tag{$\mathrm S_\alpha$}
\ee
for any smooth and compactly supported function $f$. Refer to~\cite[Chapter~5]{MR1817225} for considerations on the domain $\mathscr D_\alpha$ of~$(-\Delta)^{\alpha/2}$ based on Fourier transforms: it is the space of all tempered distributions $u$ which vanishes at infinity and such that~$(\kern-0.6pt-\Delta)^{\alpha/2}u$ is in $L^2(\R^n)$. According to~\cite[5.9(1) and 5.10(2)]{MR1817225}, we have
\[
\|(-\Delta)^{-\alpha/2}f\|_{L^2(\R^n)}^2=\frac{c_{n-2\alpha}}{c_{2\alpha}}\iint_{\R^n\times\R^n}\frac{f(x)\,f(y)}{|x-y|^{n-2\alpha}}\,dx\,dy\quad\mbox{with}\quad c_\alpha=\Gamma(\alpha)/\pi^{\alpha/2}
\]
for any $f\in L^p(\R^n)$, so that
\[
\iint_{\R^n\times\R^n}\frac{f(x)\,f(y)}{|x-y|^{n-2\alpha}}\,dx\,dy=\frac{c_{2\alpha}}{c_{n-2\alpha}}\irn{f\,(-\Delta)^{-\alpha}f}\,.
\]
Taking the Legendre transform of both sides of~\eqref{HLSgen} written for $p=2\,n/(n+2\alpha)$ and $\lambda=n-2\alpha$, we realize that inequalities ~\eqref{HLSgen} and~\eqref{Sobfrac} are equivalent because $q=2\,n/(n-2\alpha)$ is the H\"older conjugate of $p$. Moreover, the optimal constants are such that
\[
S_{n,\alpha}=\frac{c_{2\alpha}}{c_{n-2\alpha}}\,\frac1{C_{p,\lambda,n}}\quad\mbox{with}\quad\lambda=n-2\alpha\quad\mbox{and}\quad p=\frac{2\,n}{n+2\alpha}\,.
\]
See~\cite{Lieb-1983} and~\cite{Carlen-2017} for further details on duality issues. Both~\eqref{Sobfrac} and~\eqref{HLSgen} inequalities are invariant by scaling and conformally invariant.

The classical case of \emph{Sobolev's inequality} corresponds to~\eqref{Sobfrac} with $\alpha=1$. If $n\ge 3$, the inequality can be written with $S_n=S_{n,1}$ as
\be{Sob}
\nrm{\nabla f}2^2\ge S_n\,\nrm f{\frac{2n}{n-2}}^2\tag{S}
\ee
for all functions $f\in L^\frac{2n}{n-2}(\R^n)$ such that $\nabla f\in L^2(\R^n)$. Inequality~\eqref{Sob} has a long history. An early computation of the best constant and the optimizers for the Sobolev inequality among radial functions can be found in~\cite{Bliss-1930} (also see~\cite{Rosen-1971} for the linear stability of the radial optimizers). Without symmetry, the inequality is proved in~\cite{zbMATH03035784,zbMATH03212366} and the computation of the constant and the optimal functions is classically attributed to T.~Aubin~\cite{Aubin-1976} and G.~Talenti~\cite{Talenti-1976}, even if it seems that a first complete proof goes back to E.~Rodemich's seminar~\cite{Rod66} (see~\cite[p.~158]{MR1814364} for a quote). The value of the best constant $S_n$ was found to be
\[
S_n=\pi\,n\,(n-2)\(\frac{\Gamma(n/2)}{\Gamma(n)}\)^{2/n}
\]
and equality is achieved in~\eqref{Sob} if and only if $f$ is in the $(d+2)$-dimensional manifold~$\mathscr M_{\mathrm S}$ of the \emph{Aubin-Talenti functions}
\[
h_{\mu,x_0,\sigma}(x):=\mu\(1+\sigma\,|x-x_0|^2\)^{-(n-2)/2}\quad\forall\,x\in\R^d
\]
parametrised by $(\mu,x_0,\sigma)\in\R\times\R^d\times(0,+\infty)$. In terms of duality, it is straightforward to check that the~\eqref{HLSgen} inequality corresponding to~\eqref{Sob} is given by the particular choice $\lambda=n-2$ and $p=t$ of the parameters, and can be written as
\be{HLSsp}
\int_{\R^n}f\,(-\Delta)^{-1}f\,dx\le \mathcal C_n\,\|f\|_{L^\frac{2n}{n+2}(\R^n)}^2
\ee
with
\[
\mathcal C_n=\frac1{S_n}=\frac1{\pi\,n\,(n-2)}\(\frac{\Gamma(n)}{\Gamma(n/2)}\)^\frac2n\,.\tag{HLS}
\]
This case is important as it encodes the $L^p$ smoothing properties of $(-\Delta)^{-1}$ on~$\R^n$.

Inequality~\eqref{HLSgen} was introduced by G.H.~Hardy and J.E.~Littlewood~\cite{HL-1928,HL-1930,Hardy_1932} on $\R$ and generalised by S.L.~Sobolev~\cite{zbMATH03035784,zbMATH03212366} to $\R^N$. The proof of the sharp~\eqref{HLSgen} inequality is due to E.H.~Lieb in~\cite{Lieb-1983} with a more detailed proof in~\cite{MR1817225}. This proof uses rearrangement techniques. The first proof of existence of an optimal function without symmetrisation is due to P.-L.~Lions in~\cite[Theorem~2.1]{Lions_1985b} and relies on the concentration-compactness method. This proof is useful for generalizations involving non-symmetric convolution kernels, but does not allow to identify the optimal functions. Also see~\cite[Section~II.4]{MR1078018} for the application to sharp Sobolev inequalities. The concentration-compactness method is currently used, see for instance~\cite{Palatucci_2013,Bonder_2018,Zhang_2021,de_Pablo_2022}, although most papers on the topic focus on the fractional Sobolev inequality rather than on the corresponding~\eqref{HLSgen} inequalities. A new rearrangement-free proof was provided by R.L.~Frank and E.H.~Lieb in~\cite{MR2925386,MR2848628,MR2858468}.

For the sake of simplicity, we restrict this presentation to few results and give sketches of the proofs only in the case of~\eqref{Sob} and~\eqref{HLSsp}. Without pretending to any exhaustivity, we also list some results for~\eqref{Sobfrac},~\eqref{HLSgen} and some related inequalities.

\emph{Gagliardo-Nirenberg-Sobolev inequalities} refer to the two famous papers~\cite{MR102740,MR109940}. Within this class of inequalities, the inequality
\be{GNSineq}
\nrm uq\le C\,\nrm up^{1-\theta}\,\nrm{\nabla u}2^\theta\tag{GNS}
\ee
holds for all smooth compactly supported functions, and by density, for all functions $u\in L^p(\R^n)$ such that $\nabla u\in L^2(\R^n)$. We shall assume here that the parameters $p$ and $q$ are such that \hbox{$2\le p<q<2^*$} where $2^*=+\infty$ if $n=1$ or $2$, and $2^*=2\,n/(n-2)$ if $n\ge3$. The exponent
\[
\theta=\tfrac{2\,n\,(q-p)}{q\,(2\,d-p\,(d-2))}
\]
is uniquely determined by the scaling properties of the inequalities. Optimality in inequality~\eqref{GNSineq} is achieved among radial functions but there are only few cases for which the best constants and the minimizers are known, for instance if $q=2\,(p-1)$ (see~\cite{Gunson91,DD-2002,CE-N-V-2004}) or if $d=1$.

Optimality results for the \emph{logarithmic Sobolev} (log-Sobolev) inequality, the \emph{logarithmic Hardy-Little\-wood-Sobolev} (log-HLS) inequality, the \emph{Onofri} inequality or the \emph{Caffarelli-Kohn-Nirenberg} inequalities, among many other functional inequalities, have attracted a lot of attention over the years and we may refer for instance to~\cite{Gross-1975,Carlen-Loss-1992,DEL-2016} for some key papers and to~\cite[Chapter~1, bibliographical comments]{BDNS} for a short review. Functional inequalities in bounded domains or on manifolds will not be discussed here, except for a few results on the sphere $\mathbb S^n$ which are related to inequalities on $\R^n$ by the stereographic projection.

Once optimal constants are known and the set of optimising functions has been characterised, the next question is to understand \emph{stability}: which kind of distance is measured by the \emph{deficit}, that is, the difference of the two terms in the functional inequality, written with the optimal constant. A variety of answers has been obtained during the last 30 years and this is what we are now going to review, mostly in the case of~\eqref{Sob} and~\eqref{HLSsp}.

%%%%%%%%%%%%%%%%%%%%%%%%%%%%%%%%%%%%%%%%%%%%%%%%%%%%%%%%%%%%%%%%%%%%%%
%%%%%%%%%%%%%%%%%%%%%%%%%%%%%%%%%%%%%%%%%%%%%%%%%%%%%%%%%%%%%%%%%%%%%%
\section{Quantitative stability results}\label{Sec:nonconstructive-stab-Sobolev}

In the celebrated paper~\cite{BL-1985}, H.~Brezis and E.H.~Lieb raise the question of the quantitative stability for Sobolev inequalities and gave an answer in bounded domains. A related result is proved by H.~Egnell, F.~Pacella, and M.~Tricarico in~\cite{EPT-1989}. In the case of the Sobolev inequality in the whole Euclidean space, the first quantitative stability result is obtained by G.~Bianchi and H.~Egnell in~\cite{Bianchi-Egnell-1991}. They prove that
\be{BE}
\nrm{\nabla f}2^2- S_n\,\nrm f{\frac{2n}{n-2}}^2\ge\kappa_{\mathrm S}\inf_{h\in\mathscr M_{\mathrm S}}\nrm{\nabla f-\nabla h}2^2
\ee
for some positive constant $\kappa_{\mathrm S}$ whose value is not known, as they argue by compactness using the concentration-compactness method and by contradiction. Stability results are of course not limited to~\eqref{Sob}. In recent years, the problem of proving stability for various sharp inequalities related to~\eqref{Sob} in analysis and geometry, such as the isoperimetric inequality, the Brunn-Minkowski inequality, the logarithmic Sobolev inequality, \emph{etc.}, has been widely studied: see for instance~\cite{FMP10,FMP13,FI13,FJ17}. The stability of the Sobolev inequality in the case of an $L^q$ norm of the gradient with $q\neq 2$ is proved by A.~Cianchi, N.~Fusco, F.~Maggi and A.~Pratelli in~\cite{CFMP09} using mass transportation, and improved recently by A.~Figalli, R.~Neumayer and Y.~R.-Y.~Zhang in~\cite{FN18,Neu20,FZ20}.

The proof in~\cite{Bianchi-Egnell-1991} uses the Hilbertian nature of $H^1(\R^d)$. When the Hilbertian nature of the underlying spaces is lost, as this is the case for the Gagliardo-Nirenberg-Sobolev inequalities, the same argument cannot be used. To circumvent this difficulty two approaches are used for~\eqref{GNSineq}. The first approach relies on a trick due to D.~Bakry, that can be found in~\cite{BGL-2014} (see a detailed computation in~\cite[Section~1.3.1.2]{BDNS}): the~\eqref{GNSineq} inequality with $q=2\,(p-1)$, whose optimal functions are generated by
\be{fp}
\mathsf u_q(x):=\big(1+|x|^2\big)^{-1/(q-1)}\quad\forall\,x\in\R^d\,,
\ee
is rewritten as a Sobolev inequality in a higher, artificial, dimension such that the optimal functions appear as Aubin-Talenti functions. This idea is used first by E.~Carlen and A.~Figalli in~\cite{Carlen-Figalli-2013} and later expanded by F.~Seuffert in~\cite{Seu17} and V.H.~Nguyen in~\cite{Ngu19}. Constants are anyway non-constructive as they rely on~\cite{Bianchi-Egnell-1991}. Starting with~\cite{Struwe-1984}, another approach (which will not be reviewed here) has been developed for studying the stability of the critical points of~\eqref{Sobfrac} in $H^{-1}$: see~\cite{aryan2021stability,Ciraolo_2017,deng2021sharp,Figalli_2020,MR4348297, Wei-Wu}. Results are again non-constructive and we are not aware of any counterpart in the case of~\eqref{HLSgen}.

There is a different line of thought which relies on entropy methods: see~\cite[Chapter~1]{BDNS} for results based on the calculus of variations. Notice here that stability is measured in terms of a relative Fisher information, which is not the same notion of distance as in~\eqref{BE} but via the Pinsker-Csisz\'ar-Kullback inequality it controls a distance equivalent to $\nrm{|f|^p-|h|^p}1$. Up to minor restrictions, these estimates can be made constructive and will be listed in the next section.

The Bianchi-Egnell method is not limited to the case $\alpha=1$ of~\eqref{Sobfrac}. In~\cite[Theorem~1]{MR3179693}
S.~Chen, R.L.~Frank and T.~Weth prove a stability result for~\eqref{Sobfrac}: for all~\hbox{$f\in\mathscr D_\alpha\subset L^q(\R^n)$} with $q=2n/(n-2\alpha)$ and $\alpha\in(0,n/2)$,
\be{CFW}
\nrm{(-\Delta)^{\alpha/2}f}2^2-S_{n,\alpha}\,\nrm fq^2\tag{$\mathrm S_\alpha$}\ge\kappa_{\mathrm S,\alpha} \inf_{h\in\mathscr M_{\mathrm S,\alpha}}\,\nrm{(-\Delta)^{\alpha/2}(f-h)}2^2
\ee
where the manifold $\mathscr M_{\mathrm S,\alpha}$ of optimal functions is generated from
\[
h(x)=\big(1+|x|^2\big)^\frac{2\alpha-n}2
\]
by multiplications by a constant, translations and scalings. The computation of the spectrum of the linearized problem uses a reformulation on the sphere which will be illustrated in Section~\ref{Sec:local-stab-HLS}. The dual counterpart for~\eqref{HLSgen} inequalities is~\cite[Theorem~1.5]{Carlen-2017} due to E.~Carlen, which goes as follows.
%---------------------------------------------------------------------
\begin{thm}\label{HLSCarlen} Let $n\ge2$. For all $\alpha\in(0,n/2)$, there is a constant $\kappa_{{\mathrm{HLS}},\alpha}>0$ depending only on $n$ and $\alpha$ such that for all $f\in L^{2n/(n+2\alpha)}(\R^n)$,
\be{HLSB3}
\nrm f{\frac{2\,n}{n+2\alpha}}^2-S_{n,\alpha}\,\nrm{(-\Delta)^{-\alpha/2}f}2^2\ge\kappa_{{\mathrm{HLS}},\alpha}\inf_{h\in\mathscr M_{{\mathrm{HLS}},\alpha}}\nrm{f-h}{\frac{2\,n}{n+2\alpha}}^2\,.
\ee
\end{thm}
%---------------------------------------------------------------------
Here $\mathscr M_{{\mathrm{HLS}},\alpha}$ is the manifold of all optimal functions, which is generated from 
\[
h(x)=\big(1+|x|^2\big)^{-\frac{n+2\alpha}2}
\]
by multiplication by a constant, translations and scalings. Using a duality argument to relate~\eqref{BE} with~\eqref{HLSB3} is not as straightforward as proving the equivalence of~\eqref{Sobfrac} with~\eqref{HLSgen} when $p=t$ and $\lambda=n-2\alpha$. The interplay between stability bounds and quantitative convexity estimates is crucial. This is analysed in~\cite{Carlen-2017}, with the additional motivation of giving theoretical grounds for the stability bound associated with the Keller-Lieb-Thirring inequality. This inequality estimates the fundamental eigenvalue $\lambda(V)$ of a Schr\"odinger operator $-\Delta + V(x)$ with $\nrm Vt$ for an appropriate~$t$, and can be seen as the Legendre transform of~\eqref{GNSineq} written with $p=2$: see~\cite{MR3177378} for further details. Coming back to Theorem~\ref{HLSCarlen}, an interesting consequence of the method is that $\kappa_{{\mathrm{HLS}},\alpha}$ is explicitly computable in terms of $\kappa_{\mathrm S}$. However, none of the two constants is known nor has been given an explicit estimate for them. Notice that a scheme of a direct proof based on a result of M.~Christ in~\cite{christ2014sharpened} has been proposed in~\cite{MR3429269} by H.~Liu and A.~Zhang in the case of the Heisenberg group.

We refer to~\cite{Fig13} for consequences of stability results. The remainder term in~\eqref{BE} is quadratic and as R.~Frank points out in~\cite{Frank-2021} (see also \cite{Bianchi-Egnell-1991} for a comment on this), the power two is optimal since it is not possible to bound it from below with terms like $\nrm{\nabla f}2^{2-\alpha}\,\nrm{\nabla f-\nabla h}2^\alpha$ for some $\alpha<2$. Similar issues are addressed for instance in~\cite{CFMP09,FMP13,FN18,Neu20,FZ20}. In the same spirit, on the sphere $\mathbb S^n$ with uniform probability measure $d\mu$, R.L.~Frank shows in~\cite[Theorem~2]{Frank-2021} that the optimal value for which there exists some $\kappa>0$ for which
\[\textstyle
\nrmsd{\nabla f}2^2-\frac d{q-2}\(\nrmsd fq^2-\nrmsd f2^2\)\ge\kappa\,\frac{\(\nrmsd{\nabla f}2^2-\frac d{q-2}\(\nrmsd f2^2-\bar f^2\)\)^\frac\alpha2}{\(\nrmsd{\nabla f}2^2+\frac d{q-2}\,\nrmsd f2^2\)^\frac{\alpha-2}2}
\]
with $\bar f:=\int_{\mathbb S^n}f\,d\mu$, is $\alpha=4$. Here we use the notation
\[
\nrmsd fq=\(\int_{\mathbb S^n}|f|^q\,d\mu\)^\frac2q\,.
\]
Optimal remainder terms are also found to be quartic on the non-trivial manifold $\mathbb S^1({\scriptstyle 1/\sqrt{n-2}})\times\mathbb S^{n-1}$ as the minimizer is degenerate in the sense that there is a zero mode of the Hessian of the minimisation problem that does not come from symmetries of the set of minimizers. Notice that other results can be obtained for different choices of the distance on the sphere: see, \emph{e.g.},~\cite{DEL17}. 

Summarizing, in quantitative stability results obtained from methods \emph{\`a la} Bianchi and Egnell, the exponent in the distance term is at least partially understood. Essentially nothing is known on the value of the constant $\kappa_{\mathrm S}$ and similar constants in related stability inequalities, or even how to get such \emph{constructive} estimates for strong notions of distances. There are some results for weaker notions of distance and this is what we shall expose in the next sections.

%%%%%%%%%%%%%%%%%%%%%%%%%%%%%%%%%%%%%%%%%%%%%%%%%%%%%%%%%%%%%%%%%%%%%%
%%%%%%%%%%%%%%%%%%%%%%%%%%%%%%%%%%%%%%%%%%%%%%%%%%%%%%%%%%%%%%%%%%%%%%
\section{Stability results for \texorpdfstring{\eqref{HLSsp}}{HLS} by parabolic methods}\label{Sec:stab-HLS}

E.~Carlen, J.A.~Carrillo and M.~Loss prove in~\cite{CCL-2010} that
\[
\mathcal C_n\,\|f\|_{L^\frac{2n}{n+2}(\R^n)}^2 - \int_{\R^n}f(x)(-\Delta)^{-1}f(x)\,dx=\frac8{n+2}\int_0^{+\infty}e^{\beta t}\,\mathcal D\big[u^\frac{n-1}{n+2}\,\big(\cdot\,,\,e^{\beta t}\big)\big]\,dt\,,
\]
where $u=u(t,x)$ is the solution of the \emph{fast diffusion} equation
\[
\frac{\partial u}{\partial t}=\Delta\,u^m\quad\mbox{in}\quad \R^n\,,\quad t\ge0\,,
\]
with exponent $m=n/(n+2)$ and initial datum $u(0,x)=f(x)$. The deficit is measured by
\[
\mathcal D[g]:=\mathcal C_n\,\frac{n\,(n-2)}{(n-1)^2}\,\nrm g2^2\,\nrm g{p+1}^{2(p-1)}-\nrm g{2p}^{2p}
\]
for any function $g\in H^1(\R^n)$, where $p=(n+1)/(n-1)$. By~\cite{DD-2002}, it turns out that $\mathcal D[g]\ge0$ is a sharp form of~\eqref{GNSineq} inequalities, with equality if and only if $g(x)=\big(1+|x|^2\big)^{1/(p-1)}$ up to a multiplication by a constant, a translation and a scaling. The proof of the equality case can be achieved either by symmetrisation, variational methods and ODE techniques, or using the \emph{carr\'e du champ} method adapted to the fast diffusion equation (see~\cite{MR3497125} for detailed justifications and references). This result provides us with a proof of the sharp~\eqref{HLSgen} inequality, identifies the optimal function in terms of a Barenblatt profile and provides an integral deficit term. Moreover, the \emph{carr\'e du champ} method allows to bypass symmetrisation techniques.

Another fast diffusion equation, now with exponent $m=(n-2)/(n+2)$, provides us with similar results. If $n\ge3$, a formal computation shows that
\[
\frac d{dt}\(\mathcal C_n\,\|f\|_{L^\frac{2n}{n+2}(\R^n)}^2-\int_{\R^n}f(x)(-\Delta)^{-1}f(x)\,dx\)=2\,\|f\|_{L^\frac{2n}{n+2}(\R^n)}^\frac4{n+2}\,\mathcal D[v]
\]
where $v=u^m$ and, for any $g\in L^\frac{2n}{n-2}(\R^n)$ such that $\nabla g\in L^2(\R^n)$,
\[
\mathcal D[g]:=\|\nabla g\|_{L^2(\R^n)}^2-S_n\,\|g\|_{L^\frac{2n}{n-2}(\R^n)}^2\,.
\]
In the spirit of the \emph{carr\'e du champ} method, it can be shown that the inequality $\mathcal D[g]$ is monotone non-increasing with limit $0$, so that $\mathcal D[g]\ge0$ is in fact the standard Sobolev inequality, with optimal constant $S_n=1/\mathcal C_n$. In practice the solution of the fast diffusion equation with $m=(n-2)/(n+2)$ vanishes for some finite time $T>0$ and justifications require some additional care. For instance, it is convenient to use the inverse stereographic projection and rewrite the evolution equation on the unit sphere: see~\cite{1101,Jin_2014,MR3227280} for details. Altogether, we may conclude that
\[
\mathcal C_n\,\|f\|_{L^\frac{2n}{n+2}(\R^n)}^2 - \int_{\R^n}f(x)(-\Delta)^{-1}f(x)\,dx=2\int_0^T\mathcal D\big[u^m(t,\cdot)\big]\,dt\,.
\]

In both cases associated with $m=n/(n+2)$ and $m=(n-2)/(n+2)$, it is not known how to express the deficit of the inequality in terms of the initial datum of the evolution equation.

%%%%%%%%%%%%%%%%%%%%%%%%%%%%%%%%%%%%%%%%%%%%%%%%%%%%%%%%%%%%%%%%%%%%%%
%%%%%%%%%%%%%%%%%%%%%%%%%%%%%%%%%%%%%%%%%%%%%%%%%%%%%%%%%%%%%%%%%%%%%%
\section{Constructive stability results}\label{Sec:constructive-stab-Sobolev}

Few constructive stability results are known so far. In~\cite{FMP13}, A.~Figalli, F.~Maggi and A.~Pratelli use mass transportation and rearrangements to prove constructive stability results for the $1$-Sobolev anisotropic inequalities for functions of bounded variation. The initial idea to prove those results comes from a previous work of N.~Fusco, F.~Maggi and A.~Pratelli, who prove a non-constructive stability result for the isotropic $1$-Sobolev inequality for functions of bounded variation.

As a straightforward consequence of the duality approach, an explicit stability bound for~\eqref{Sobfrac} is easily obtained where the distance to the Aubin-Talenti manifold $\mathscr M_{\mathrm S,\alpha}$ is measured in terms of the deficit in~\eqref{HLSgen} with $\lambda=n-2\alpha$. Let us give some details. We consider the case of~\eqref{HLSgen} with the particular choice $\lambda=n-2\alpha$ and $p=t=2\,n/(n+2\alpha)$ of the parameters. The optimal function for~\eqref{HLSgen} is, up to a multiplication by a constant, a translation and a scaling, given by
\[
g_\star(x):=\big(1+|x|^2\big)^{\frac\lambda2-n}\quad\forall\,x\in\R^d\,.
\]
We know from~\cite{Lieb-1983} and~\cite[Theorem~1.2]{Carlen-2017} that the optimal functions for~\eqref{Sobfrac} is $g_\star^r$ with $r=\frac{n-2\alpha}{n+2\alpha}$, up to multiplications by a constant, translations and scalings. By a standard trick used for instance in~\cite[Theorem~1.2]{1101},~\cite[Theorem~1.4]{Jin_2014} and~\cite[Theorem~1,~(i)]{jankowiak2014fractional}, the expansion of the square in
\[
0\le\int_{\R^n}\left|\nrm fq^{\frac{4\alpha}{n-2\alpha}}\nabla(-\Delta)^{\frac{\alpha-1}2}f-S_{n,\alpha}\,\nabla(-\Delta)^{-\frac{1+\alpha}2}g\right|^2\,dx\,,
\]
with $g^r=f$ so that $\nrm fq^2=\nrm gp^{2r}$ if $q=2\,n/(n-2\alpha)$, shows that
\begin{multline*}\label{eq:mainresult}
S_{n,\alpha}\,\(\nrm gp^2-S_{n,\alpha}\,\nrm{(-\Delta)^{-\alpha/2}g}2^2\)\\
\le \nrm fq^{\frac{8\alpha}{n-2\alpha}}\(\nrm{(-\Delta)^{\alpha/2}f}2^2-S_{n,\alpha}\,\nrm fq^2\)\,.
\end{multline*}
Slightly better results are obtained using flow methods as in~\cite{1101,MR3227280,jankowiak2014fractional} which have not been detailed in Section~\ref{Sec:stab-HLS}. Notice that the expansion of the square applies to logarithmic inequalities corresponding to the limit cases as $\alpha\to n/2$. See~\cite[Theorem~10]{Dolbeault_2022} for further considerations in the case $\alpha=1$ and $n=2$. The main drawback is that, in all these approaches, the stability is controlled only in a weaker norm.

The situation is slightly better for subcritical~\eqref{GNSineq} inequalities. Constructive results are obtained by M.~Bonforte \emph{et al.} in~\cite{BBD+09,BDGV10} in a very small neighborhood of the manifold of the Aubin-Talenti functions and were later improved in~\cite{Dolbeault2011a,DT-2013}. The global result of~\cite{DT-2013} is explicit but sub-optimal as the remainder term is of the order of the square of the entropy while one expects a linear dependence in view of~\cite{BDGV10,Dolbeault2011a}. It is true also in the logarithmic case as explained in~\cite{DT16c} through scaling properties. Notice that stability results in logarithmic Sobolev inequalities  is a widely studied question, with various results in $L^1$ and in Wasserstein distances: see~\cite{Fathi_2016,Feo_2016,die034-0708-437,kim2021instability} and references therein. The proof of~\cite{DT-2013} is based on the use of the remainder terms in the \emph{carr\'e du champ} method and has an already long history in the framework of entropy methods for linear diffusion equations: see~\cite{MR2152502,0528}. A remarkable property of the \emph{carr\'e du champ} method is that it also applies to nonlinear flows of fast diffusion type~\cite{Carrillo2000,Carrillo2003,DT-2013,1501,DEL-JEPE}, eventually with a non-trivial metric (see~\cite{BGL-2014} and references therein), on the Euclidean space or on manifolds: see for instance~\cite{Demange_2008,Dolbeault20141338}. As a consequence, let us give some examples on the sphere, which are taken from~\cite{DEKL,Dolbeault_2020b,DEL17}. Using the uniform probability measure $d\mu$ as in Section~\ref{Sec:nonconstructive-stab-Sobolev}, for any $q\in(1,2)\cup(2,2^*)$ and any $u\in H^1(\mathbb S^d)$,
\[
\nrmsd{\nabla u}2^2-\tfrac d{q-2}\(\nrmsd uq^2-\nrmsd u2^2\)\ge d\,\psi\(\tfrac{\nrmsd uq^2-\nrmsd u2^2}{(q-2)\,\nrmsd uq^2}\)\,\nrmsd uq^2
\]
holds for an explicit, strictly convex $C^2$ function $\psi$ such that $\psi(0)=\psi'(0)=0$ (see~\cite{Dolbeault_2020b} for details). Here equality is achieved only by constant functions, but optimality constants are determined by perturbations involving the spherical harmonics associated with the first eigenvalue of the Laplace-Beltrami operator on $\mathbb S^n$. By requesting additional orthogonality constraints, which discard these spherical harmonics, the constant $d$ in the left-hand side can be improved (with $\psi=0$). See~\cite[Section~5]{DEL17} and~\cite{Dolbeault_2020b}. The corresponding \emph{improved entropy -- entropy production} inequality can be reinterpreted as a stability result. The constant $\kappa$ in front of the distance term is obtained through a classical minimisation problem. The value of the minimum is not explicitly known, except in the limit case of the logarithmic Sobolev inequality as $q\to2$ or if additional symmetry assumptions are imposed.

Recent work by M.~Bonforte \emph{et al.} provides the first quantitative and constructive stability result for~\eqref{GNSineq} and Sobolev inequalities on the Euclidean space, under some constraints. Using entropy methods and fast diffusion flows, the authors prove in~\cite{BDNS} (also see~\cite{BDNS2022} for a scheme of the proof and further considerations on the method) that stability for initial data can be deduced from improved decay rates of relative entropies in rescaled variables when the rescaling is chosen in order to match the solution with the best matching Barenblatt profile among all Barenblatt self-similar solutions: the stability is then measured by a relative Fisher information and takes the form
\[
\inf_{\varphi\in\mathscr M_{\mathrm{GNS}}}\irn{\left|(q-1)\,\nabla u+u^q\,\nabla\varphi^{1-q}\right|^2}\,,
\]
where $\mathscr M_{\mathrm{GNS}}$ is the manifold generated from $\mathsf u_q$ defined in~\eqref{fp} by multiplications by a constant, translations and scalings. The result holds for any of the inequalities~\eqref{GNSineq} with $q=2\,(q-1)\in(1,2^*)$ and also for the Sobolev inequality~\eqref{Sob} when $q=2^*$, $d\ge3$. The method relies on regularisation properties of the fast diffusion flows which introduce an integral decay condition on the initial data. Regularity and properties of the entropy are then used to extend the asymptotic stability results into the \emph{initial time layer}, where nonlinear evolution takes place and the evolution of the entropy is controlled using the nonlinear \emph{carr\'e du champ} method by backward in time estimates. The core of the method is a constructive Harnack inequality based on J.~Moser's methods and a fully quantitative global Harnack Principle for the nonlinear flow. This quantifies the \emph{threshold time} after which the solution is in the \emph{asymptotic regime} and the convergence is governed by an improved Hardy-Poincar\'e inequality, based on a spectral analysis. Improved decay rates are then extended to the initial time layer, thus proving an improved entropy -- entropy production inequality for the solution, which is also valid for the initial data. The improved inequality is rephrased as a stability result while Barenblatt profiles are transformed into the Aubin-Talenti type functions~$\mathsf u_q$. The whole method relies on entropies, which requires to work with functions with a finite second moment, and on a global Harnack Principle, which provides us with a uniform threshold time if and only if the tails of the initial data have sufficient decay. These limitations are the price to pay in order to get a constructive stability estimate with an explicit constant.

Establishing constructive stability results for~\eqref{HLSsp} is so far an open question.

%%%%%%%%%%%%%%%%%%%%%%%%%%%%%%%%%%%%%%%%%%%%%%%%%%%%%%%%%%%%%%%%%%%%%%
%%%%%%%%%%%%%%%%%%%%%%%%%%%%%%%%%%%%%%%%%%%%%%%%%%%%%%%%%%%%%%%%%%%%%%
\section{Two local stability results for \texorpdfstring{\eqref{HLSsp}}{HLS}}\label{Sec:local-stab-HLS}

%%%%%%%%%%%%%%%%%%%%%%%%%%%%%%%%%%%%%%%%%%%%%%%%%%%%%%%%%%%%%%%%%%%%%%
\subsection{A local stability result in the norm of relative uniform convergence}\label{Sec:local-stab-HLS-RUC}

In the spirit of~\cite{BBD+09,MR3227280}, we first establish a \emph{constructive} stability result of~\eqref{HLSsp} in a neighbourhood of the optimal functions, with respect to the very strong topology of relative uniform convergence. See~\cite{BDNS} for an illustration of the interest of this framework. The price to pay is that the result applies in a neighbourhood of~$\mathscr M:=\mathscr M_{{\mathrm{HLS}}_1}$ in a topology stronger than the natural topology.

Equality in~\eqref{HLSsp} is achieved by the function
\[
u_\star(x)=\big(1+|x|^2\big)^{-(n+2)/2}\quad\forall\,x\in\R^d\,.
\]
Let us define the functions
\[
f_0:=u_\star^\frac{n-2}{n+2}=\big(1+|x|^2\big)^{-\frac{n-2}2}\,,\quad f_i :=\frac{x_i}{1+|x|^2}\,f_0\,,\quad f_{n+1}:=\frac{1-|x|^2}{1+|x|^2}\,f_0\,,
\]
and the norm
\be{star-norm}
\|v\|_\star^2:=\irn{u_\star^{-\frac4{n+2}}\,v(x)^2}=\irn{v(x)^2\,\big(1+|x|^2\big)^2}\,.
\ee
For perturbations $u_\star+\varepsilon\,u_\star^\frac4{n+2}\,g$ of $u_\star$ such that
\be{BoundRUC}
\big\|\,u_\star^{-\frac{n-2}{n+2}}\,g\,\big\|_{L^\infty(\R^n)}\le1\,,
\ee
there is a stability result for~\eqref{HLSsp} which is uniform in $\varepsilon>0$ small enough. Notice that~\eqref{BoundRUC} can be rewritten as
\[
(1-\varepsilon)\,u_\star\le u_\varepsilon:=u_\star+\varepsilon\,u_\star^\frac4{n+2}\,g\le(1+\varepsilon)\,u_\star\,.
\]
This means that $u_\varepsilon/u_\star$ is $\varepsilon$-close to $1$, \emph{i.e.}, $u_\varepsilon$ is close to $u_\star$ in the topology of \emph{relative uniform convergence} (see~\cite{BBD+09,BDNS}) associated with the norm $v\mapsto\nrm{v/u_\star}\infty$.
%---------------------------------------------------------------------
\begin{thm}\label{Thm:RUC} Let $n\ge3$. If $g$ fullfils~\eqref{BoundRUC} and the orthogonality conditions
\be{orthogg}
\irn{\frac{g\,f_i}{\big(1+|x|^2\big)^2}}=0\quad\forall\,i=0,1,\dots,n+1\,,
\ee
then for any $\varepsilon\in(0,1)$, the function $u_\varepsilon=u_\star+\varepsilon\,u_\star^\frac4{n+2}\,g$ satisfies
\[
\mathcal C_n\,\nrm u{\frac{2\,n}{n+2}}^2-\irn{u\,(-\Delta)^{-1}u}\ge\kappa_n\,\|u-u_\star\|_\star^2
\]
with $\kappa_n:=\frac{8\,(n+1)}{3\,n\,(n+2)^2\,(n+4)}$.
\end{thm}
%---------------------------------------------------------------------
\begin{proof} Let $p=2\,n/(n+2)$. By using the fact that $u_\star$ is optimal, inequality~\eqref{HLSsp} can be written as $\mathscr H[u]\ge0$ with
\begin{multline*}
\mathscr H[u]:=\,\mathcal C_n\(\nrm up^2-\nrm{u_\star}p^2-2\,\nrm{u_\star}p^\frac4{n+2}\irn{u_\star^\frac{n-2}{n+2}\,(u-u_\star)}\)\\
-\,\irn{(u-u_\star)\,(-\Delta)^{-1}(u-u_\star)}\,.
\end{multline*}
Assume that $g$ satisfies~\eqref{BoundRUC} and~\eqref{orthogg}. We notice that
\[
\irn{u_\star^\frac{n-2}{n+2}\,(u-u_\star)}=\irn{f_0\,(u-u_\star)}=0\,,
\]
and as a consequence
\[
\mathscr H[u]=\mathcal C_n\(\nrm up^2-\nrm{u_\star}p^2\)-\irn{(u-u_\star)\,(-\Delta)^{-1}(u-u_\star)}\,.
\]
By a Taylor-Lagrange expansion, we have
\begin{multline*}
|u|^p-u_\star^p-p\,u_\star^{p-1}\,(u-u_\star)-\tfrac12\,p\,(p-1)\,u_\star^{p-2}\,(u-u_\star)^2\\
=\tfrac16\,p\,(p-1)\,(p-2)\,|\xi|^{p-4}\,\xi\,(u-u_\star)^3
\end{multline*}
for some intermediate value $\xi$ between $u_\star$ and $u$. With $p=2\,n/(n+2)\in(1,2)$, $u_\star>0$ and $|u-u_\star|<\varepsilon\,u_\star$ for some $\varepsilon\in(0,1)$, we know that $\xi$ is positive. In that case, either $u<u_\star$ and the right-hand side is positive, or $u\ge u_\star$ and
\[
|\xi|^{p-4}\,\xi\,(u-u_\star)^3\le u_\star^{p-3}\,(u-u_\star)^3\,.
\]
By using $|g|\le u_\star^\frac{n-2}{n+2}$ by~\eqref{BoundRUC}, the expansion applied with $u=u_\varepsilon$ shows that
\[
\tfrac16\,p\,(p-1)\,(p-2)\,|\xi|^{p-4}\,\xi\,(u-u_\star)^3\ge-\frac{4\,n\,(n-2)}{3\,(n+2)^3}\,\varepsilon^3\,u_\star^\frac4{n+2}\,g^2\,.
\]
After taking into account~\eqref{orthogg} and $u_\star^\frac4{n+2}(x)=\big(1+|x|^2\big)^{-2}$, we obtain
\[
\irn{|u_\varepsilon|^p}\ge\irn{u_\star^p}+\frac{n\,(n-2)}{(n+2)^2}\,\varepsilon^2\(1-\frac{4\,\varepsilon}{3\,(n+2)}\)\irn{u_\star^\frac4{n+2}\,g^2}\,.
\]
By a Taylor-Lagrange expansion again, we know that $s^q\ge s_\star^q+q\,{s_\star}^{q-1}\,(s-s_\star)$ if $q>1$ and $s>s_\star>0$. Applied with $q=1+2/n$, $s=\irn{|u_\varepsilon|^p}$ with $\varepsilon\in(0,1)$ so that $\frac{4\,\varepsilon}{3\,(n+2)}<1$, and $s_\star=\irn{u_\star^p}$, we infer that
\[
\nrm{u_\varepsilon}p^2-\nrm{u_\star}p^2\ge\(1-\frac{4\,\varepsilon}{3\,(n+2)}\)\,\varepsilon^2\,\mathcal M[g]
\]
where
\[
\mathcal M[g]:=\frac{n-2}{n+2}\,\nrm{u_\star}p^\frac4{n+2}\irn{\frac{g(x)^2}{\big(1+|x|^2\big)^2}}\,.
\]
With
\[
\mathcal D[g]:=\irn{\frac{g(x)}{\big(1+|x|^2\big)^2}\,(-\Delta)^{-1} \frac{g(x)}{\big(1+|x|^2\big)^2}}\,,
\]
we know that
\[
\mathscr H[u_\varepsilon]\ge\mathcal C_n\,\mathcal M[g]\,\varepsilon^2\(1-\frac{4\,\varepsilon}{3\,(n+2)}-\frac{\mathcal D[g]}{\mathcal C_n\,\mathcal M[g]}\).
\]
Finally, we study the quotient $\mathcal D[g]/\mathcal M[g]$ among functions which satisfy the orthogonality conditions~\eqref{orthogg}. To this end, let us consider the orthonormal basis of functions of $L^2\big(\R^n,(1+|x|^2)^{-2}dx\big)$ made of the spherical harmonics on $\mathbb S^n$ mapped into $\R^n$ via the stereographic projection. They satisfy the equations
\[
-\Delta g_k=\mu_k\,\frac{g_k}{\big(1+|x|^2\big)^2}\,,\quad \mu_k=4\,k(k+n-1)+ n\,(n-2)
\]
and as a consequence, we have
\[
\mathcal D[g_k]=\frac1{\mu_k}\,\irn{\frac{|g_k(x)|^2}{\big(1+|x|^2\big)^2}}\,.
\]
Under condition~\eqref{orthogg}, $g=\sum_{k\ge2}a_k\,g_k$ is such that
\[
\irn{\frac{g(x)^2}{\big(1+|x|^2\big)^2}}=\sum_{k\ge2}|a_k|^2\quad\mbox{and}\quad \mathcal D[g]=\sum_{k\ge2}\frac{|a_k|^2}{\mu_k}\,.
\]
Since $(\mu_k)_{k\in\N}$ is an increasing sequence, it follows that 
\be{D}
\frac{\mathcal D[g]}{\irn{\frac{g(x)^2}{\big(1+|x|^2\big)^2}}}\le\frac1{\mu_2}=\frac1{(n+2)\,(n+4)}
\ee
Taking into account the fact that
\[
\nrm{u_\star}p^p=\frac{2^{1-n}\,\pi^\frac{n+1}2}{\Gamma(\frac{n+1}2)}
\]
and the expression of~$\mathcal C_n$ in~\eqref{HLSsp}, we obtain
\[
\mathscr H[u_\varepsilon]\ge\frac{4\,\varepsilon^2}{n\,(n+2)}\(\frac1{n+4}-\frac\varepsilon{3\,(n+2)}\)\ge\kappa_n\,\varepsilon^2\irn{\frac{g(x)^2}{\big(1+|x|^2\big)^2}}\,,
\]
because $\varepsilon<1$. This completes the proof of Theorem~\ref{Thm:RUC}.
\end{proof}

A stability result similar to Theorem~\ref{Thm:RUC} can be established for~\eqref{HLSgen} with the particular choice $\lambda=n-2\alpha$ and $p=t=2\,n/(n-2\alpha)$ of the parameters, $\alpha\in(0,n/2)$, by adapting the computations of~\cite[Section~4]{jankowiak2014fractional}.

\medskip Let us consider on $H^{-1}(\R^n)$ the norm $\nnrm f=\nrm{\nabla(-\Delta)^{-1}f}2$ such that
\[
\nnrm f^2=\irn{f\,(-\Delta)^{-1}f}=\pi^{2-\frac n2}\,\Gamma(n-2)\iint_{\R^n\times\R^n}\frac{f(x)\,f(y)}{|x-y|^{n-2}}\,dx\,dy\,.
\]
Here we consider again the case $\lambda=n-2$ of~\eqref{HLSgen}, with $p=t=2\,n/(n+2)$, and recall that $L^p(\R^n)\subset H^{-1}(\R^n)$. For any $f\in L^p(\R^n)$ and any $h\in\mathscr M:=\mathscr M_{{\mathrm{HLS}},1}$, let us define the quotient
\[
\mathscr Q[f,h]:=\frac{\nrm fp^2-S_n\,\nrm{(-\Delta)^{-1/2}f}2^2}{\nnrm{f-h}^2}\,.
\]
Our goal is to prove that for some $\kappa>0$
\be{LocalStab}
\inf_{h\in\mathscr M}\mathscr Q[f,h]\ge\kappa
\ee
for all functions in a small neighbourhood of $\mathscr M$, to be defined. Theorem~\ref{Thm:RUC} is already a result of this type: if
\[
g=\frac1\varepsilon\,(f-u_\star)\,u_\star^{-\frac4{n+2}}
\]
satisfies~\eqref{BoundRUC} for some $\varepsilon\in(0,1)$, using~\eqref{D}, we obtain a first \emph{local stability result} with $\kappa=(n+2)\,(n+4)\,\kappa_n$. An important drawback is that the topology of the \emph{relative uniform convergence} is very strong so that the eligible functions $f$ is a small set in the natural function space.

%%%%%%%%%%%%%%%%%%%%%%%%%%%%%%%%%%%%%%%%%%%%%%%%%%%%%%%%%%%%%%%%%%%%%%
\subsection{A local stability result in a weighted norm}

Here we want to work in the more natural framework of the norm defined by~\eqref{star-norm}. We recall that $p=2\,n/(n+2)$. By H\"older's inequality, we can write
\be{Holder2}
\nrm fp^2\le\nrm{u_\star}p^\frac4{n+2}\,\|f\|_\star^2\,,
\ee
which means that $\|f\|_\star^2$ is a stronger norm than $\nrm fp^2$. Equality holds in~\eqref{Holder2} for $f>0$ such that $\nrm fp=\nrm{u_\star}p$ if and only if $f=u_\star$ so that
\[
\nrm fp^2-\nrm{u_\star}p^{\frac4{n+2}}\,\|f\|_\star^2
\]
measures the distance of $f$ to $u_\star$. This can be made precise using $(1+s)^{2/p}\ge1+2\,s/p$ for any $s\ge0$ and, according to the generalised Pinsker-Csisz\'ar-Kullback inequality (see \emph{e.g.},~\cite[Proposition~1.1]{BDIK}) as follows: for any function $f\in L^p(\R^n)$,
\begin{align*}
\(\nrm fp^2-\nrm{u_\star}p^2\)&\ge\frac2p\,\nrm{u_\star}p^{2-p}\(\irn{|f|^p}-\irn{u_\star^p}\)\\
&\ge2^{1-\frac2p}\,(p-1)\,\nrm{|f|-u_\star}p^2\,.
\end{align*}

For simplicity, we shall assume again that $\alpha=1$, so that $\lambda=n-2$ as in Section~\ref{Sec:local-stab-HLS-RUC}. The extension to $\alpha\neq1$ is left to the reader. Since $u_\star$ is a critical point of $\mathscr H[f]$, we can write that
\[
1+\mathscr Q[f,u_\star]=\mathcal C_n\,\frac{\nrm fp^2-\nrm{u_\star}p^2-2\,\nrm{u_\star}p^\frac4{n+2}\irn{u_\star^\frac{n-2}{n+2}\,(f-u_\star)}}{\nnrm{f-u_\star}^2}
\]
%---------------------------------------------------------------------
\begin{thm}\label{Thm:star} Let $n\ge3$. Let $f\in L^p(\R^n)$ be a nonnegative function such that the orthogonality conditions~\eqref{orthogg} are fullfiled by
\[
g=(f-u_\star)\,u_\star^{-\frac4{n+2}}\,.
\]
If one has
\be{ConditionK}
\frac{n+4}{3\,(n+2)}\,(1+\eta)\irn{u_\star^{p-3}\,|f-u_\star|^3}\le\irn{u_\star^{p-2}\,|f-u_\star|^2}
\ee
for some $\eta>0$, then~\eqref{LocalStab} holds with $\kappa=\frac{4\,\eta}{n\,(1+\eta)}$.
\end{thm}
%---------------------------------------------------------------------
\begin{proof} Let us consider $h$ such that
\[
f-u_\star=\varepsilon\,h\quad\mbox{with}\quad\varepsilon:=\nnrm{f-u_\star}
\]
and define
\[
\mathsf X:=\|h\|_\star^2=\irn{u_\star^{p-2}\,h^2}\quad\mbox{and}\quad\mathsf Y:=\irn{u_\star^{p-3}\,|h|^3}\,.
\]
A Taylor-Lagrange expansion shows that
\[
\irn{|f|^p}\ge\irn{u_\star^p}+\frac{n\,(n-2)}{(n+2)^2}\,\varepsilon^2\,\mathsf X-\frac{4\,n\,(n-2)}{3\,(n+2)^3}\,\varepsilon^3\,\mathsf Y\,.
\]
Using
\[
(1+s)^{2/p}\ge1+\frac{2\,s}p
\]for any $s\ge0$, we obtain
\[
\nrm{f}p^2-\nrm{u_\star}p^2\ge K_0\,\varepsilon^2\(\frac{n\,(n-2)}{(n+2)^2}\mathsf X-\frac{4\,n\,(n-2)}{3\,(n+2)^3}\,\varepsilon\,\mathsf Y\)
\]
with
\[
K_0:=\frac 2p\,\nrm{u_\star}p^{2-p}=\pi\,\frac{n+2}n\(\frac{\Gamma(n/2)}{\Gamma(n)}\)^\frac2n=\frac1{n^2}\,\frac{n+2}{n-2}\,\frac1{\mathcal C_n}\,.
\]
On the other hand, Inequality~\eqref{D} applied to $g=\big(1+|x|^2\big)^2(f-u_\star)$ shows that
\be{lll}
\irn{(f-u_\star)\,(-\Delta)^{-1}(f-u_\star)}\le\frac{\|f-u_\star\|_\star^2}{(n+2)\,(n+4)}=\frac{\varepsilon^2\,\mathsf X}{(n+2)\,(n+4)}\,.
\ee
Hence
\begin{align*}
\mathscr H[f]&=\mathcal C_n\(\nrm{f}p^2-\nrm{u_\star}p^2\)-\irn{(f-u_\star)\,(-\Delta)^{-1}(f-u_\star)}\\
&\ge\frac1{n^2}\,\frac{n+2}{n-2}\,\varepsilon^2\(\frac{n\,(n-2)}{(n+2)^2}\,\mathsf X-\frac{4\,n\,(n-2)}{3\,(n+2)^3}\,\varepsilon\,\mathsf Y\)-\frac{\varepsilon^2\,\mathsf X}{(n+2)\,(n+4)}\\
&\hspace*{12pt}=\frac{4\,\varepsilon^2}{n\,(n+2)}\(\frac{\mathsf X}{n+4}-\frac{\varepsilon\,\mathsf Y}{3\,(n+2)}\)\ge
\frac{4\,\varepsilon^2}{n\,(n+2)\,(n+4)}\,\frac\eta{1+\eta}\,\mathsf X\,.
\end{align*}
By~\eqref{lll}, we have $\mathsf X\ge(n+2)\,(n+4)$, which concludes the proof.
\end{proof}

With the notations of the proof of Theorem~\ref{Thm:star}, Condition~\eqref{ConditionK} can be rewritten as $\mathsf X-\mathscr K\,\mathsf Y\ge0$ with $\mathscr K=\frac{n+4}{3\,(n+2)}\,\varepsilon\,(1+\eta)$. This somewhat unusual condition has two implications:
\begin{itemize}
\item[(1)] $h\in L^p(\R^n)$ is bounded in the stronger norm $\big(\irn{u_\star^{p-3}\,|h|^3}\big)^{1/3}$, which is however a much weaker condition than~\eqref{BoundRUC},
\item[(2)] the function $h$ is limited to an explicit neighbourhood of $0$, in strong norms.
\end{itemize}
A detailed statement goes as follows.
%---------------------------------------------------------------------
\begin{prop} If $h$ is a function in $L^p(\R^n)$ such that $\irn{u_\star^{p-3}\,|h|^3}<+\infty$, which satisfies
\be{stab1}
\irn{u_\star^{p-2}\,h^2}-\mathscr K\,\irn{u_\star^{p-3}\,|h|^3}\ge0
\ee
for some $\mathscr K>0$, then $h$ also has the following properties:
\begin{subequations}
\begin{align}
&\label{stab2X}
\irn{u_\star^{p-2}\,h^2}\le\(\mathscr K^{-1}\,\nrm hp^\frac p{2-p}\)^{2-p}\,,\\
&\label{stab2Y}
\irn{u_\star^{p-3}\,|h|^3}\le\mathscr K^{-1}\(\mathscr K^{-1}\,\nrm hp^\frac p{2-p}\)^{2-p}\,,\\
&\label{stab3}\nrm hp^\frac p{2-p}\le\(\mathscr K^{-1}\,\nrm{u_\star}p\)^\frac p{2-p}\,.
\end{align}
\end{subequations}
\end{prop}
%---------------------------------------------------------------------
\begin{proof}

With $\mathsf a=\(\irn{|h|^p}\)^\frac1{2-p}$, $\mathsf X:=\irn{u_\star^{p-2}\,h^2}$, $\mathsf Y:=\irn{u_\star^{p-3}\,|h|^3}$ as in the proof of Theorem~\ref{Thm:star}, the inequality
\[
\irn{u_\star^{p-2}\,h^2}\le\(\irn{u_\star^{p-3}\,|h|^3}\)^\frac{2-p}{3-p}\(\irn{|h|^p}\)^\frac1{3-p}=\(\mathsf a\irn{u_\star^{p-3}\,|h|^3}\)^\frac{2-p}{3-p}\,,
\]
can be rephrased as 
\[
\mathsf X\le(\mathsf a\,\mathsf Y)^\frac{2-p}{3-p}\quad\mbox{or}\quad\mathsf Y\ge\frac1{\mathsf a}\,\mathsf X^\frac{3-p}{2-p}\,.
\]
The stability assumption~\eqref{stab1} means that
\begin{align*}
&0\le\mathsf X-\mathscr K\,\mathsf Y\le\mathsf X-\frac{\mathscr K}{\mathsf a}\,\mathsf X^\frac{3-p}{2-p}=\mathsf X\(1-\frac{\mathscr K}{\mathsf a}\,\mathsf X^\frac1{2-p}\)\,,\\
&0\le\mathsf X-\mathscr K\,\mathsf Y\le(\mathsf a\,\mathsf Y)^\frac{2-p}{3-p}-\mathscr K\,\mathsf Y=\mathsf Y^\frac{2-p}{3-p}\(\mathsf a^\frac{2-p}{3-p}-\mathscr K\,\mathsf Y^\frac1{3-p}\)\,,
\end{align*}
which proves $\mathsf X\le\(\mathsf a\,\mathscr K^{-1}\)^{2-p}$ and $\mathsf Y\le\mathscr K^{-1}\(\mathsf a\,\mathscr K^{-1}\)^{2-p}$, \emph{i.e.},~\eqref{stab2X} and~\eqref{stab2Y}.

By H\"older's inequality, we obtain
\[
\mathsf a^{2-p}=\irn{|h|^p}=\irn{\(u_\star^\frac{p-2}2\,|h|\)^p\,u_\star^{p\,\frac{2-p}2}}\le\(\irn{u_\star^{p-2}\,h^2}\)^\frac p2\(\irn{u_\star^p}\)^{1-\frac p2}
\]
with $\mathsf b:=\nrm{u_\star}p$. After taking into account~\eqref{stab2X}, this amounts to $\mathsf a\ge\(\frac{\mathsf b}{\mathscr K}\)^\frac p{2-p}$, which concludes the proof of~\eqref{stab3}.
\end{proof}

%%%%%%%%%%%%%%%%%%%%%%%%%%%%%%%%%%%%%%%%%%%%%%%%%%%%%%%%%%%%%%%%%%%%%%
%%%%%%%%%%%%%%%%%%%%%%%%%%%%%%%%%%%%%%%%%%%%%%%%%%%%%%%%%%%%%%%%%%%%%%
\begin{ack} The authors thank R.L.~Frank for pointing~\cite{Rod66} to them. They also thank an anonymous referee for very relevant and precise remarks.\end{ack}
\begin{funding} This research has been partially supported by the projects \emph{EFI}~ANR-17-CE40-0030 (J.D.) and \emph{molQED} (M.J.E.) of the French National Research Agency (ANR).\\
\copyright\,2022 by the authors. This paper may be reproduced, in its entirety, for non-com\-mercial purposes.\end{funding}
%%%%%%%%%%%%%%%%%%%%%%%%%%%%%%%%%%%%%%%%%%%%%%%%%%%%%%%%%%%%%%%%%%%%%%
\begin{spacing}{0.9}
\bibliographystyle{emss}
\bibliography{Lieb90}
\end{spacing}
\end{document}